\documentclass[aoas,preprint,12pt]{article}
\RequirePackage[OT1]{fontenc}
\RequirePackage{amsthm,amsmath}
\RequirePackage[numbers]{natbib}
\RequirePackage[colorlinks,citecolor=blue,urlcolor=blue]{hyperref}
\usepackage{amsthm,amsmath,natbib}
\usepackage{amsfonts}
\usepackage{amssymb}
\usepackage{graphicx}
\newtheorem{theorem}{Theorem}[section]
\newtheorem{lemma}{Lemma}[section]

\newtheorem{proposition}{Proposition}[section]

\newtheorem{definition}{Definition}[section]
\numberwithin{equation}{section}
\newcommand{\bnu}{{\boldsymbol \nu}}
\newcommand{\bxi}{{\boldsymbol \xi}}
\newcommand{\bet}{{\boldsymbol \eta}}

\newcommand{\bSigma}{{\boldsymbol \Sigma}}

\def\by{{\boldsymbol y}}
\def\bw{{\boldsymbol w}}
\def\bA{{\boldsymbol A}}
\def\bm{{\boldsymbol m}}
\def\bY{{\boldsymbol Y}}
\def\bz{{\boldsymbol z}}
\def\be{{\boldsymbol e}}

\def\bZ{{\boldsymbol Z}}

\def\bg{{\boldsymbol g}}
\def\bx{{\boldsymbol x}}
\def\bX{{\boldsymbol X}}

\graphicspath{%
    {converted_graphics/}
    {/}
}
\begin{document}
\title{Best predictors in logarithmic distance \\between positive random variables} 
\author{Henryk Gzyl\\
\noindent 
Centro de Finanzas IESA, Caracas, (Venezuela)\\
 henryk.gzyl@iesa.edu.ve}

\date{}
 \maketitle

\setlength{\textwidth}{4in}

\vskip 1 truecm
\baselineskip=1.5 \baselineskip \setlength{\textwidth}{6in}
\begin{abstract}
The metric properties of the set in  which random variables take their values lead to relevant probabilistic concepts. For example, the mean of a random variable is a best predictor in that it minimizes the standard Euclidean distance or $L_2$ norm in an appropriate class of random variables. Similarly, the median is the same concept but when the distance is measured by the $L_1$ norm.\\ 
It so happens that a {\it geodesic distance} can be defined on the cone of strictly positive vectors in $\mathbb{R}^n$ in such a way that the minimizer of the distance to a collection of points is their geometric mean.\\
This distance induces a distance on the class of strictly positive random variables, which in turn leads to an interesting notions of conditional expectation (or best predictors) and their estimators. The appropriate version of the Law of Large Numbers and the Central Limit Theorem, can also be obtained.  We shall see that, for example,  the lognormal variables are the analogue of the Gaussian variables for the modified version of the Central Limit Theorem. 
\end{abstract}


\section{Introduction and Preliminaries} 
The study of random variables and processes taking values in spaces with geometries other than Euclidean in not new. Consider the textbooks by Kunita and Watanabe  \cite{KW}  or by Hsu \cite{Hsu} to mention just two. Along this line of work, the notion of Euclidean distance between points of the base manifold is replaced by a distance related to a Riemannian metric placed upon the tangent manifold. Such metrics lead to a notion of geodesic distance between points of the manifold, and such distance is inherited by random variables taking values in the manifold.

It should not then be surprising that the notion of best predictor of a random variable by variables of a given class, should depend on the metric of the manifold. In this note we shall consider the manifold to be $M=(0,\infty)^N,$ which is an open set in $\mathbb{R}^n,$ which is also a commutative group with respect to component wise multiplication. We postpone the study of the geometry of this group to the appendix. Here we mention that what we do is the commutative version of a more elaborate geometry in the space of symmetric matrices. The reader can check with Lang \cite{Lang} in which a relation of this geometry to Bruhat-Tits spaces is explained, or in Lawson and Lim \cite{LL} or Mohaker \cite{Moh} and references therein, where the geometric mean property in the class of symmetric matrices is established. More recently Resigny et al. \cite{AFPA} and Schwartzman \cite{Sch} used the same geometric setting to study the role of such geometry in a large variety of applications. The applications of the geometric ideas in these references concern the non-commutative case, but the simplest commutative case and its potential usefulness for positive random variables seems not to have been explored.

As mentioned in the abstract, it is the purpose of this note to explore the possible usefulness of measuring distances between positive numbers, not by regarding them as real numbers and the distance between them measured by the Euclidean norm, but by a logarithmic distance resulting from an interesting group invariant metric.

The appendix is devoted to basic geometry. There we shall examine the geometry on $M$ and prove that the distance between any two points $\bx_i,\bx_2\in M$ is given by
\begin{equation}\label{dist2}
d(\bx_1,\bx_2)^2 = \sum_{i=1}^n \left(\ln x_1(i)-\ln x_2(i)\right)^2.
\end{equation}
This makes $M$ a Tits-Bruhat space in which the distance satisfies a semi-parallelogram law. This is contained in Theorem \ref{lang}. We shall use this property to establish the uniqueness of conditional expectations. And the group structure in $M$ will be inherited in a curious way by the conditional expectations (or by the best predictors) in the logarithmic distance (\ref{dist1}).

But once we have motivated the appearance of the logarithmic distance, and the semi-parallelogram law associated to it, we shall come to the main objective of the paper, which is to consider the notion of best predictor
 (conditional expectations) in that distance, which happens to have some curious properties. These matters will be taken up in Sections 2 and 3, where we shall introduce the notion of $\ell-$expected value and $\ell-$conditional expectation, which will denote the best predictors in the logarithmic distance (hence the $\ell-$prefix) introduced in Section 2. We examine there some of the basic properties of these constructs.

In Section 4 we present the two most basic estimators, namely, that of the $\ell-$mean and that of the $\ell-$variance, and explain how the law of large numbers and the central limit theorem for these estimators relates to the standard law of large numbers and the central limit theorems.

In section 5 we prove that the notion of martingale related to the $\ell-$conditional expectation relates to the standard notion of martingale. We shall do it in discrete time, but the extension to continuous time is quite direct. 
In Section 6 we examine Markowitz portfolio theory when the distance between (gross) returns is the logarithmic distance.  

As said, we leave the study of the geometry on $M$ to the appendix. There we explain how the logarithmic distance between strictly positive vectors is actually a geodesic distance in that manifold. For that we shall present some results from Lang's \cite{Lang}, but in a simpler, commutative setup. This will provide us with a way of thinking about positive numbers (or vectors) in terms of the exponential map. The aim of the section is to derive the logarithmic distance between positive vector as a geodesic distance. The basic idea behind our constructions has been very much studied in geometry. The vectors with non-zero components act transitively on the positive vectors in such a way that an invariant scalar product (a Riemannian metric) can be defined which leads to a notion of geodesic distance.  Actually, the exponential function will correspond to the exponential map in Riemannian geometry, and it will allow us to relate (transport) probabilistic constructs from the real to the positive numbers (vectors)

\section{Best predictors in logarithmic distance}
Our set up here consists of a probability space $(\Omega, \mathcal{F},P)$ and we shall be concerned with the cone $\mathcal{C}$ of $P-$almost everywhere (a.e. for short) finite and strictly positive ($M$-valued) random variables. As usual, we identify variables that are $P-$a.e. equal. Since the operations among vectors are component wise, to reduce to the case $n=1$ only takes a simple notational change. To shorten the description of the random variables used in the statements coming up below, let us introduce the following notations. For $p>1$ (we shall be concerned with $p=1,2$ only) define:
$$L_p = \{\bX\in\mathcal{F}\,|\, E[|X_i|^p]<\infty,\;\;i=1,...,n\}$$
$$Ln_p = \{\bX\in \mathcal{C}\,|\,\ln\bX \in L_p\},\qquad LLn_p=L_p\bigcap Ln_p.$$
Let $\bX_1$ and $\bX_2$ be two strictly positive random variables in $Ln_2$. The (logarithmic) distance between them is defined to be
\begin{equation}\label{logdis}
d_\ell(\bX_1,\bX_2)^2 \equiv E\left[\sum_{i=1}^n(\ln X_1(i) - \ln X_2(i))^2\right]
\end{equation}
Since we are identifying variables that are a.e equal, $d_\ell(\bX_1,\bX_2)$ is a distance on $\mathcal{C}.$ Similarly to $\bm = E[\bX]$ being the constant that minimizes the Euclidean (squared) distance to $\bX,$ we have
\begin{proposition}\label{logmean}
With the notations introduced above, let $\bX \in Ln_1.$ The vector $\bm_\ell$ that minimizes the logarithmic distance to $\bX$ is given by
$$\bm_\ell(\bX) = \exp(E[\ln\bX]).$$
\end{proposition}
The proof of the first assertion is computational, and the second results from an application of Jensen's inequality. When there is no risk of confusion, we shall write $\bm_\ell(\bX)=\bm_\ell.$ Keep in mind that the operations are componentwise, and that $\bm_{\ell}(\bX)_j=\exp(E[\ln X_j])$ for $j=1,...,n.$ If $\bX \in LLn_1,$  we also have $\bm_\ell \leq E[\bX].$

And the analogues of the notions of covariance and  centering are contained in the following definition.
\begin{definition}\label{logcovar}
Let now $\bX, \bY \in Ln_2.$ We define the logarithmic covariance matrix of the non-negative random variables $\bX$ and $\bY$ by
$$Cov_\ell(\bX,\bY) \equiv E[\left(\ln\bX - \ln\bm_\ell(\bX)\right)\left(\ln\bY -\ln\bm_\ell(\bY)\right)^t ] = Cov(\ln\bX,\ln\bY).$$
Let $\bSigma$ be the matrix with components $E[\left(\ln X_i - \ln m_\ell(X_i)\right)\left(\ln Y_j -\ln m_\ell(Y_j)\right)].$ If the matrix $\bSigma$ is invertible, we define the ``centered'' (in logarithmic distance) version of $\bX$ by
$$\bX^c \equiv \exp\left(\bSigma^{-1/2}\left(\ln\bX - \ln\bm_\ell(\bX)\right)\right)$$
\end{definition}
The need for the exponentiation is clear: First we have to ``undo'' the taking of the logarithms and second, the argument of the exponential function is a vector in $\mathbb{R}^n$ which yields a positive vector after exponentiation.  It takes a simple computation to verify that 
 $$\bm_\ell(\bX^c) = \mathbf{1},\;\;\;\bSigma_\ell(\bX^c) = \mathbb{I}.$$

A variation on the previous theme consists of predicting a variable $\bY$ by a variable $\bX$ in logarithmic distance. The extension of the previous result is contained in the following statement.
\begin{proposition}\label{genpred}
Let $\bY$ and $\bX$ be in $Ln_2$. Then the $\sigma(\bX)-$measurable random variable that minimizes the logarithmic distance (\ref{logdis}) to $\bY$ is given by
$$E_\ell[\bY|\bX] = \exp\left(E[\ln\bY\,|\,\bX]\right).$$
And we also have $E_\ell[\bY|\bX] \leq E[\bY\,|\,X].$
\end{proposition}
The proof of Proposition \ref{genpred} follows the same pattern as the standard proof. Just notice that $\phi(\bX)=\exp\left(E[\ln\bY\,|\,\bX]\right)$ is a bounded, $\sigma(\bX)-$measurable random variable, such that  $\ln\phi(\bX)=E[\ln\bY\,|\,X]$ minimizes the Euclidean square distance to $\ln\bY.$ 

Note that the last inequality mentioned in the statement does not mean that one of the estimators is better than the other in any sense. They are minimizers in different metrics.
Also, since linear combinations in an exponent are transported as scaling and powers, we have the following analogue to linear prediction for positive random variables.
\begin{proposition}\label{expred}
Let $Y$ and $X$ be positive real variables with square integrable logarithms. The values of $a >0$ and $b\in\mathbb{R}$ that make $Y^{\#} \equiv aX^b$ the best predictor of $Y$ in the logarithmic metric, are given by
$$
\left\{\begin{array}{l}
          a = \exp\left(E[(\ln Y)] - bE[(\ln X)]\right)\\
	  b = \frac{1}{D}\left(E[\ln X \ln Y] - E[\ln X]E[\ln Y]\right)\\
	  D = E[(\ln X)^2] - (E[\ln X])^2 = \sigma^2(\ln X).
	  \end{array}\right.
$$
\end{proposition}
The proof follows the standard computation starting from the definition of $d(\bY,a\bX^b)_\ell.$  Certainly the result is natural as the linear structure of $\mathbb{R}$ is transferred multiplicatively onto $(0,\infty)$ by the exponential mapping. Also, the extension to random variables taking values in higher dimensional $M$ is direct, but notationally more cumbersome.\\
A simple computation leads to
$$m_\ell(Y^\# )=E_\ell[Y^\#]=e^{E[ln Y]},\;\;\;\;\sigma_\ell(Y^\#)=b^2\sigma^2(\ln X).$$

\section{Logarithmic conditional expectation and some of its properties }
Here we extend the semi-parallelogram property mentioned in Theorem (\ref{lang}) to strictly positive random variables.
\begin{lemma}\label{splrv}
All random variables mentioned are supposed to be in $Ln_2.$ Let $\bX_1$ and $\bX_2$ be as mentioned. Then there exits $\bZ \in Ln_2$ such that for any $\bY$ we have
$$d(\bX_1,\bX_2)^2_\ell + 4d(\bZ,\bY)^2_\ell \leq 2d(\bY,\bX_1)^2_\ell +2d(\bY,\bX_2)^2_\ell.$$
\end{lemma}
To prove this, use the second comment after Theorem (\ref{lang}) at every $\omega\in \Omega$ to obtain the pointwise version of the semi-parallelogram property, and then integrate with respect to $P.$ Clearly $\bZ=(\bX_1\bX_2)^{1/2}\in LLn_2.$ Below we apply this to obtain the uniqueness of the extension of the standard notion of conditional expectation.	
\begin{theorem}\label{CE}
Let $\mathcal{G}\subset\mathcal{F}$ be a $\sigma-$ algebra, and let $\bY$ be non-negative with square integrable logarithm. Then, the unique -up to a set of $P$ measure $0$-, positive $\bX^*\in \mathcal{G}$ that makes $d(\bY,\bX)^2_\ell$ minimum over $\{\bX \in \mathcal{G},\;\bX > 0, E[(\ln\bX)^2] <\infty\},$ is given by $\bX^*=\exp\left(E[\ln\bY\,|\,\mathcal{G}]\right).$
To be consistent with the notations introduced above, we shall write $\bX^*=E_\ell[\bY\,|\,\mathcal{G}].$
\end{theorem}

\begin{proof}
The existence follows the same pattern of proof as the propositions in the previous section, that is 
$E[\ln\bY\,|\,\mathcal{G}]$ minimizes the ordinary  square distance to $\ln\bY,$ and it is the unique (up to sets of $P$ measure $0$). We shall use the semi-parallelogram property to verify the uniqueness. For that, let $\bX$ some other possible minimizer of the logarithmic distance. Now set $\bZ=\sqrt{\bX\bX^*}$ (keep in mind the second comment after Theorem (\ref{lang})), and observe that according to the semi-parallelogram property
$$d(\bX^*,\bX)^2_\ell + 4d(\bY,\bZ)^2_\ell \leq 2d(\bY,\bX^*)^2_\ell + 2d(\bY,\bX)^2_\ell.$$
Since by definition, $d(\bY,\bZ)^2_\ell$ is larger than any of the two distances in the right hand side of the inequality, it follows that necessarily $d(\bX^*,\bX)^2_\ell=0.$
\end{proof}

Let us now verify some standard and non standard properties of the notion of conditional expectation introduced above. Keep in mind that the arithmetic operations with positive vectors are componentwise.

\begin{theorem}\label{propCE}
 Let $\bY\in LLn_2$ and let $\mathcal{H}\subset\mathcal{G}$ be two sub-$\sigma-$algebras of $\mathcal{F}.$ Then, up to a set of measure $0,$ the following hold:\\
{\bf 1)} $E_\ell[\bY\,|\{\emptyset,\Omega\}] = E_\ell[\bY].$\\
{\bf 2)}$E_\ell[E_\ell[\bY | \mathcal{G}]\,|\mathcal{H}] = E_\ell[\bY\,|\mathcal{H}].$\\
{\bf 3)}Let $\bY_1,...,\bY_k$ be in $LLn_2,$ and $w_i \in \mathbb{R}.$ The analogue of the linearity property of the standard conditional expectation is the following multiplicative property:
$$E_\ell[\prod_{i=1}^k\bY_i^{w_i}\,|\,\mathcal{G}] = \prod_{i=1}^k\Big(E_\ell[\bY|\mathcal{G}]\Big)^{w_i}$$
{\bf 4)} If $\bY$ is independent of $\mathcal{G}$ in the standard sense, then
$E_\ell[\bY\,|\mathcal{G}] = E_\ell[\bY].$
\end{theorem}

\begin{proof}
The first assertion is simple consequence of the definition . To verify the second we start from the definition and carry on:.
$$E_\ell[E_\ell[\bY |\mathcal{G}]\,|\mathcal{H}] = \exp\Big(E\big[\ln\exp E[\ln\bY |\mathcal{G}|\mathcal{H}\big]\Big) = \exp\Big(E[E[\ln\bY |\mathcal{G}|\mathcal{H}]\Big),$$
\noindent and now apply the standard tower property of conditional expectations to the complete the proof of the assertion.\\
It is in the third property where the logarithmic distance plays a curious role. The proof of the assertion is a simple computation starting from the definition:
$$E_\ell[\prod_{i=1}^k\bY_i^{w_i}\,|\,\mathcal{G}]=\exp\Big(E\big[\sum w_i\ln\bY_i\,|\mathcal{G}\big]\Big) = \prod_{i=1}^k\Big(E_\ell[\bY|\mathcal{G}]\Big)^{w_i}.$$
The fourth property is also simple to establish using the definition and the standard notion of independence.
\end{proof}

\section{Estimators and limit theorems}
In this section we shall consider the case $n=1.$ The notation is a bit simpler in this case. That is, we shall forget about the symbols in boldface for a while.

Making use of Proposition (\ref{geomean}) the following definition is clear:
\begin{definition}\label{empmean}
Let $X_1,...,X_K$ be positive random variables. We define their empirical logarithmic mean by
$$\hat{m}_\ell(X) = \left(\prod_{j=1}^KX_j\right)^{1/K}.$$
\end{definition}
And a the standard law of large numbers becomes:
\begin{theorem}\label{LLN}
Let $X_j,\,j \geq 1$ be a collection of i.i.d. positive random variables defined on $(\Omega,\mathcal{F},P)$ having finite logarithmic variance $\sigma_\ell^2$ and mean $m_\ell.$ Then $\hat{X}_\ell$ is an unbiased estimator of the logarithmic mean $m_\ell(X)$ and
$$\hat{m}_\ell(X) = \left(\prod_{j=1}^KX_j\right)^{1/K} \rightarrow m_\ell$$
\noindent almost surely w.r.t. $P$ as $K \rightarrow \infty. $
\end{theorem}
The proof is clear. Since 
$$\hat{m}_\ell(X) = \exp\left(\frac{1}{K}\sum_{j=1}^K\ln X_j\right),$$
\noindent we can invoke the strong law of large numbers, see Borkhar \cite{Bork} or Jacod and Protter \cite{JP} , plus the continuity of the exponential function to obtain our assertion. That $\ln\hat{m}_\ell(X)$ has mean $m_\ell(X)$ is clear.

In analogy with the standard notion of empirical variance, we can introduce
\begin{definition}\label{empvar}
With the notations introduced above and under the assumptions in Theorem \ref{LLN}, the empirical estimator of the logarithmic variance is defined by
 $$\hat{\sigma}^2_\ell(X) = \frac{1}{K-1}\sum_{j=1}^K \left(\ln X_j - \ln\hat{m}_\ell(X)\right)^2.$$
\end{definition}
And as in basic statistics we have
\begin{theorem}\label{LLNempvar}
With the notations introduced above, and under the assumptions of Theorem (\ref{LLN}), $\hat{\sigma}^2_\ell(X)$ is an unbiased estimator of the logarithmic variance and 
$$\hat{\sigma}^2_\ell(X) \rightarrow \sigma_\ell^2(X)$$
\noindent almost surely w.r.t. $P$ as $K\rightarrow\infty.$
\end{theorem}
But perhaps more interesting is the following version of the central limit theorem. It brings to the fore the role of lognormal variables as the analogue to the Gaussian random variables in the class of positive variables.

\begin{theorem}\label{CLT}
Suppose that $X_j, j\geq 1$ are a collection of i.i.d. random variables defined on a probability space $(\Omega,\mathcal{F},P)$ with logarithmic mean $m_\ell = E[\ln X_j]$ and $E[(\ln X_i)^2] < \infty.$ Then
$$\Big(\prod_{j=1}^K\frac{X_j}{m_\ell}\Big)^{1/\sqrt{K}} \rightarrow e^{X}$$
\noindent in probability as $K \rightarrow \infty,$ where $X \sim N(0,\sigma^2_\ell).$
\end{theorem}
\begin{proof}
Observe that
$$\Big(\prod_{j=1}^K\frac{X_j}{m_\ell}\Big)^{1/\sqrt{K}} = \exp\Big(\frac{1}{\sqrt{K}}\sum_{j=1}^K(\ln X_j-\ln m_\ell)\Big).$$
From the standard proof of the central limit theorem we know that  $\frac{1}{\sqrt{K}}\sum_{j=1}^K(\ln X_j-\ln m_\ell)$ converges in probability to an $N(0,\sigma^2_\ell)$ random variable and therefore, since the exponential function is continuous, the same convergence holds for $\left(\prod_{j=1}^K\frac{X_j}{m_\ell}\right)^{1/\sqrt{K}}. $ Thus concludes the proof of our assertion.
\end{proof}

\section{$\ell-$martingales in discrete time}
As there is a notion of $\ell-$conditional expectation, there must be a corresponding notion of $\ell-$ martingale. In this section we examine some very simple of its properties. As usual, the basic setup consists of the probability space $(\Omega,\mathcal{F},P)$ and a filtration $\{\mathcal{F}_n,\,n \geq 0\}.$
\begin{theorem}\label{mart}
The $M-$valued process $\{\bX_n;n \geq 0\}$ such that $\bX_n\in \mathcal{F}_n$ and $\bxi_n=\ln\bX_n$ are square integrable, is an $\ell$-martingale (resp. sub-martingale, super-martingale) if and only if $\{\bxi_n\}$is an ordinary martingale.

Also, if $\bX_n$ is an $\ell-$martingale, it is an ordinary sub-martingale.
\end{theorem}
\begin{proof}For $n\geq 0$ and $k \geq 1$
$$E_\ell[\bX_{n+k}|\mathcal{F}_n] = e^{E[\bxi_{n+k}|\mathcal{F}_n]}$$
\noindent from which the assertion of the theorem drops out. For the second assertion note that
$$E_\ell[\bX_{n+k}|\mathcal{F}_n] = \bX_n = e^{\bxi_n} = e^{E[\bxi_{n+k}|\mathcal{F}_n]} \leq E[e^{\bxi_{n+k}}|\mathcal{F}_n] = E[\bX_{n+k}|\mathcal{F}_n]$$
The middle step drops out from Jensen's inequality.
\end{proof}

The corresponding version of the Doob decomposition theorem, say for sub-martingales, goes as follows.
\begin{theorem}\label{DD}
With the notations introduced above, let $\{\bX_n\}$ be an $M-$valued $\ell-$sub-martingale. Then there exist an $M-$valued $\ell-$martingale $\{\bY_n\}$ and an increasing $M-$valued process $\bA_n,$ such that $\bX_n=\bY_n\bA_n.$
\end{theorem}
\begin{proof}
Just apply the Doob decomposition theorem to $\bxi_n=\ln\bX_n$ and use $\bX_n = e^{\bxi_n}.$ 
\end{proof}
 
\section{Logarithmic geometry and portfolio theory}
 Let us introduce a slight change of notation to conform with the notation is standard financial modeling. By the generic $R$ we shall denote the (gross) return of any asset of portfolio, which means the quotient of its current value divided by its initial value.  

To begin with, recall from (\ref{geo2}) that the curve $R_1^{w}R_2^{1-w}$ is a geodesic in the logarithmic distance between the points $R_1$ and $R_2.$ That curve can be thought of as a weighted geometric mean of $R_1$ and $R_2.$ This remark leads to variation on the theme of ``return'' of a portfolio. In our setup, a generic portfolio, characterized by the weights $w_1,...,w_K$ of assets with gross returns $R_1,...,R_k,$ has a weighted return given by $\prod_{i=1}^KR_i^{w_i}.$ To push the geodesic interpretation a bit further, that geometric mean can be thought of as a sequence of geodesic walks joining say $R_1$ to $R_K.$ Anyway, the logarithm of the $\ell-$mean, 
\begin{equation}\label{logret}
\ln m_l = \sum_{i=1}^K w_iE[\ln R_i]
\end{equation}
\noindent is clearly the logarithmic rate of growth of the portfolio. Recall as well that the logarithmic distance of $m_\ell$ to $\prod_{i=1}^KR_i^{w_i}$ is given by
\begin{equation}\label{logcov}
d(\prod_{i=1}^KR_i^{w_i},m_\ell)^2 = Var(\sum_{i=1}^K w_i\ln R_i) =(\bw,\bSigma\bw).
\end{equation}
Imitating Markowitz's portfolio theory, we assign to any portfolio $\bw$ its logarithmic mean $\bm_\ell(\bw)$ and its logarithmic variance $\sigma_\ell(\bw).$ According to Markowitz's proposal a portfolio is optimal when it minimizes the variance for a given expected value of its (rate of) return. 

The content of the following proposition can be read in two ways. On one hand it provides a prescription for a choice of portfolio with given average geometric rate of return and minimal logarithmic covariance. On the other hand, it establishes a relationship between that choice of portfolio and the choice according to the Markowitz's proposal based on the logarithmic rate of return.  
\begin{proposition}\label{mark}
With the notations introduced above, the weights $w_i^*,....,w_K^*$ that make the logarithmic variance,  $\sigma_\ell(\bw) = d(\prod_{i=1}^KR_i^{w_i},m_\ell)^2$ minimal subject to the constraints $\sum w_i=1$ and $m_\ell(\bw) = e^{\mu},$ are the same as the weights that minimize $Var\Big(\sum_{i=1}^kw_i\ln R_i\Big)$ subject to $E[\sum_{i=1}^kw_i\ln R_i]=\mu$ and $\sum w_i=1.$
\end{proposition}
The proof is clear from (\ref{logcov}).
We refer the interested reader to Luenberger (\cite{Lue} or to Shiryaev \cite{Shir} for more details about the classical Markowitz portfolio optimization theory. 

\section{Concluding comments}
In this note we proposed an alternative metric in the set of positive vectors, so that when distance between random variables is measured in this new metric, the standard notions of best predictors, their estimation, some classical convergence results, acquire a different but intuitively related form.

Also, as a simple application to finance, when assets are characterized by their gross returns (which by definition are positive random variables), the concept of return of a portfolio becomes a weighted geometric average, and the standard portfolio choice methodology appears in a slightly different guise. Readers familiar with the basics of the methodology will find it clear that the analogue of the efficient frontier, market portfolio, market line and CAPM have a counterpart within the formalism developed above, but this is not the place to pursue the matters.

\section{Appendix: The logarithmic distance between positive vectors}
We shall think of the vectors in $\mathbb{R}^n$ as functions $\bxi:\{1,...,n\}\rightarrow \mathbb{R},$ and all standard arithmetical operations  either as component wise operations among vectors or point wise operations among functions. Let us denote by $M=\{\bx\in\mathbb{R}^n\,|\bx(i)>0, i=1,...n\}$ the set of all positive vectors. $M$ is an open set in $\mathbb{R}^n$ which is trivially a manifold over $\mathbb{R}^n,$ having $\mathbb{R}^n$ itself as tangent space at each point. We shall use the standard notation $TM_{\bx}$ to stress this point.

Here $M$ plays the role that the positive definite matrices play in the works by Lang, Lawson and Lim and Mohaker mentioned a few lines above. The role of the group of invertible matrices in the same references is to be played here by $G=\{\bg\in\mathbb{R}^n\,|\,g(i)\not=0,\,i=1,...,n\},$ which  clearly is an Abelian group respect to the standard product, in which the identity, denoted by $\be,$ is the vector with all components equal to $1.$  We shall make use the action $G:M\rightarrow M$ of $G$ on $M$ defined by $\tau_{\bg}(\bx)=\bg^{-1}\bx\bg^{-1}.$ This action is clearly transitive on $M,$ and can be defined in the obvious way as an action on $\mathbb{R}^n.$ 

The transitivity of the action allows us to transport the scalar product on $TM_{\be}$ to any $TM_{\bx}$ as follows. The scalar product between $\bxi$ and $\bet$ at $TM_{\be}$ is defined to be the standard Euclidean product $(\bxi,\bet)=\sum \xi_i\eta_i,$ where we shall switch between $\xi(i)$ and $\xi_i$ as need be. Since $\bx = \tau_{\bg}(\be)$ with $\bg=\bx^{-1/2}.$ We define the scalar product transported to $TM_{\bx}$ by
$$(\bxi,\bet)_{\bx} \equiv (\bx^{-1}\bxi,\bx^{-1}\bet) = (\bx^{-2}\bxi,\bet).$$
This scalar product allows us to define the length of a differentiable curve as follows: 

Let $\bx(t)$ be a differentiable curve in $M,$ its length is given by
$$\int_0^1\sqrt{(\dot{\bx},\dot{\bx})_{\bx}}dt.$$
With this definition, the distance between $\bx_1, \bx_2 \in M$ is defined by the expected
\begin{equation}\label{dist1}
d(\bx_1,\bx_2) = \inf\{\int_0^1\sqrt{(\dot{\bx},\dot{\bx})_{\bx})}dt\,|\,\bx(t)\;\;\mbox{differentiable such that}\;\;\bx_1=\bx(0)\;\; \bx_2=\bx(1)\}
\end{equation}
It takes an application of the Euler-Lagrange formula to see that the equation of the geodesics in this metric is
\begin{equation}\label{geo1}
\ddot{\bx}(t) = \bx^{-1}\dot{\bx}^2, \;\;\bx(0)=\bx_1,\;\;\bx(1)=\bx_2,
\end{equation}
\noindent the solution to which is 
\begin{equation}\label{geo2}
\bx(t) = \bx_1e^{t\ln(\bx_2/\bx_1)} = \bx_2^{t}\bx_1^{(1-t)}.
\end{equation}
This allows us to compute the distance between $\bx_1$ and $\bx_2$ as
\begin{equation}\label{dist3}
d(\bx_1,\bx_2)^2 = \sum_{i=1}^n \left(\ln x_1(i)-\ln x_2(i)\right)^2.
\end{equation}
Similarly, the solution to (\ref{geo1}) subject to $\bx(0)=\bx,$ and $\dot{\bx}(0)=\bxi$ is the  (exponential) mapping $\bx e^{t\bxi}.$ With this notations we recall some results (in this simpler setup) from Chapter 5 of Lang (1995) under
\begin{theorem}\label{lang}
With the notations introduced above we have:\\
{\bf 1)} The exponential mapping is metric preserving through the origin.\\
{\bf 2)} The derivative of the exponential mapping is measure preserving, that is, $\exp'(\bxi)\bnu =\bnu e^{\bxi}$ as a mapping $TM_{\bx}\rightarrow TM_{\exp{\bx}},$ satisfies
$$(\bnu,\bnu) = (\exp'(\bxi)\bnu,\exp'(\bxi)\bnu)_{\exp(\bxi)}$$
{\bf 3)} With the metric given by (\ref{dist2}), $M$ is a Bruhat-Tits space, that is it is a complete metric space in which the semi-parallelogram law holds. This means that, given any $\bx_1,\,\bx_2\in M,$ there exists a unique $\bz\in M$ such that for ant $\by\in M$ the following holds
$$d(\bx_1,\bx_2)^2+4d(\bz,\by)^2 \leq 2d(\by,\bx_1)^2 + 2d(\by,\bx_2)^2.$$
\end{theorem}
{\bf Comments} \\
{\bf 1)} The action $\tau_{\bg}$ defined a few paragraphs above coincides with parallel transport along geodesics.\\
{\bf 2)}The proofs take some space but are systematic and computational. In our case, commutativity makes things considerably simpler. The completeness of $M$ is transferred from $\mathbb{R}^n$ via the exponential mapping.\\
{\bf 3)} The point $\bz$ mentioned in item (3) is given by $\bz=\sqrt{\bx_1\bx_2}.$ Actually, a simple calculation provides the proof of the following slightly more general statement.
\begin{lemma}\label{geomean}
Let $\bx_1,...,\bx_K$ be $K$ points in $M.$ The point $\bar{\bx}_\ell$ that minimizes the sum of logarithmic distances (\ref{dist2}) to the given points is given by their geometric mean, that is
$$\bar{\bx}_\ell = \left(\prod_{j=1}^K\bx_j\right)^{1/K}$$
\end{lemma}

\end{document}